\documentclass{amsart}
\usepackage{amsmath,amsthm,amssymb,amscd,tikz-cd}
\theoremstyle{plain}

\newtheorem{thm}{Theorem}[section]
\newtheorem{prop}[thm]{Proposition}

\theoremstyle{definition}
\newtheorem{ex}[thm]{Example}
\newtheorem{defn}[thm]{Definition}
\newtheorem{remark}[thm]{Remark}

\newtheorem{conj}[thm]{Conjecture}
\newtheorem*{conj*}{Conjecture}

\newcommand{\F}{{\mathbb F}}

\newcommand{\N}{{\mathbb N}}
\newcommand{\Z}{{\mathbb Z}}
\newcommand{\Q}{{\mathbb Q}}

\newcommand{\Fp}{{\mathbb F}_p}

\newcommand{\Fpbar}{\overline{{\mathbb F}}_p}

\newcommand{\Qpbar}{\overline{{\Q}}_p}

\newcommand{\Qbar}{\overline{{\Q}}}

\newcommand{\Zp}{\Z_p}

\newcommand{\Qp}{\Q_p}

\newcommand{\rhobar}{\overline{\rho}}

\DeclareMathOperator{\GL}{GL}
\DeclareMathOperator{\SL}{SL}

\DeclareMathOperator{\Gal}{Gal}

\DeclareMathOperator{\Ind}{Ind}
\DeclareMathOperator{\ind}{ind}

\DeclareMathOperator{\new}{new}
\DeclareMathOperator{\nr}{nr}
\DeclareMathOperator{\crys}{crys}

\usepackage[normalem]{ulem}
\usepackage{enumitem}

\newcommand\robout{\bgroup\markoverwith {\textcolor{blue}{\rule[0.5ex]{2pt}{0.4pt}}}\ULon} 

\newcommand\johnout{\bgroup\markoverwith {\textcolor{magenta}{\rule[0.5ex]{2pt}{0.4pt}}}\ULon}

\usepackage{hyperref}

\title[Slopes of modular forms and reducible Galois representations]{Slopes of modular forms and reducible Galois representations\\ \small An oversight in the ghost conjecture}

\author{John Bergdall and Robert Pollack}

\address{John Bergdall\\Bryn Mawr College\\ Bryn Mawr, PA 19010\\USA}
\email{jbergdall@brynmawr.edu}
\urladdr{http://jbergdall.digital.brynmawr.edu}

\address{Robert Pollack\\Department of Mathematics and Statistics \\ Boston University \\ 111 Cummington Mall \\ Boston, MA 02215\\USA}
\email{rpollack@math.bu.edu}
\urladdr{http://math.bu.edu/people/rpollack}


\numberwithin{equation}{section}

\begin{document}

\maketitle
\setcounter{tocdepth}{1}

\begin{abstract}
The ghost conjecture, formulated by this article's authors, predicts the list of $p$-adic valuations of the non-zero $a_p$-eigenvalues (``slopes") for overconvergent $p$-adic modular eigenforms in terms of the Newton polygon of an easy-to-describe power series (the ``ghost series''). The prediction is restricted to eigenforms whose Galois representation modulo $p$ is reducible on a decomposition group at $p$. It has been discovered, however, that the conjecture is not formulated correctly. Here we explain the issue and propose a salvage.
\end{abstract}


\section{Introduction}

Let $p$ be a prime number. The authors previously made a conjecture on the $p$-adic slopes of modular eigenforms with a fixed Galois representation modulo $p$ {\em subject to} that representation being locally reducible at $p$. See Conjecture \ref{conj:old-ghost} below and also \cite[Conjecture 7.1]{BergdallPollack-GhostPaper-II}. The local reducibility condition has turned out to be too strong --- one reducible situation (and its twists) must also be excluded. This article will explain the exclusion and how to salvage the conjecture. We provide both computational and theoretical evidence for the corrected conjecture. 

It is important to correct false conjectures. In this case, however, it is specifically important because there is ongoing progress toward a proof of the conjecture, for Galois representations that are locally reducible with sufficiently generic Serre weights, by Liu, Truong, Xiao, and Zhao. A paper explaining their strategy is available as a preprint \cite{LTXZ-LocalGhost}. The proof itself has also been announced \cite{LTXZ-GhostProof}. 

A general proof, however, will remain out of reach until faulty hypotheses are removed! Addressing the prior misconception further allows us the opportunity to address, in writing, the special case of eigenforms that are twists of $E_2$ modulo $p$. We had previously not tested that case out of convenience. 

In \cite{BergdallPollack-GhostPaper-II}, the authors called the reducibility assumption ``Buzzard regular'', based on Buzzard's work \cite{Buzzard-SlopeQuestions}. It is unfortunate that Buzzard's name was attached, by us, to the faulty condition. Our strongest apologies. 

\subsection{Recollection of the ghost conjecture}
\label{sec:gc}

Fix algebraic closures $\Qbar$ of $\Q$, $\Qpbar$ of $\Qp$, and $\Fpbar$ of $\Fp$. Let $G_{\Q} = \Gal(\Qbar/\Q)$ and $N \geq 1$ be an integer that is co-prime to $p$. We then consider continuous, semi-simple, representations $\rhobar : G_{\Q} \rightarrow \GL_2(\Fpbar)$ that are modular of level $N$. Recall that if $f$ is a cuspidal, normalized, eigenform of level $\Gamma_1(N)$ with coefficients in $\Qpbar$, then there is a continuous, semi-simple, Galois representation $\rhobar_f : G_{\Q} \rightarrow \GL_2(\Fpbar)$ such that if $\ell \nmid Np$ is a prime, then $\rhobar_f$ is unramified at $\ell$ and the trace of a Frobenius element at $\ell$ is equal to $a_\ell(f)$, the $\ell$-th Fourier coefficient of $f$. Saying $\rhobar$ is modular of level $N$ means $\rhobar \cong \rhobar_f$ for such an $f$.

For $k\geq 2$ we define $S_{k}(\rhobar) \subseteq S_k(\Gamma_1(N),\overline\Q_p)$ to be the $\overline \Q_p$-span of the eigenforms $f$ of weight $k$ and level $\Gamma_1(N)$ such that $\rhobar_f \cong \rhobar$. So, $\rhobar$ is modular of level $N$ exactly when $S_k(\overline \rho) \neq \{0\}$ for some $k$. Define $\Gamma_0 = \Gamma_1(N) \cap \Gamma_0(p)$ and then $S_k^0(\rhobar) \subseteq S_k(\Gamma_0,\overline \Q_p)$ in a similar fashion.  We write the dimensions of these spaces as:
\begin{align*}
d(k,\rhobar) &= \dim_{\overline\Q_p} S_{k}(\rhobar);\\
d_p(k,\rhobar) &= \dim_{\overline\Q_p} S_{k}^0(\rhobar).
\end{align*}
Set $d_p^{\new} = d_p - 2d$. (The function $d_p^{\new}$ counts the eigenforms that are $p$-new.)

The ghost conjecture is based on the definition of a sequence of monic polynomials $g_{1}(w), g_{2}(w),\dotsc \in \Zp[w]$ as follows:
\begin{enumerate}
\item \label{a} The zeros of $g_{i}(w)$ are $w_k = (1+2p)^k-1$, for the finite list of integers $k$ such that 
$$
d(k,\bar\rho) < i < d(k,\bar\rho)+d^{\new}(k,\bar\rho).
$$
\item Fix $k$ now. If $i_1,i_2,\dotsc,i_m$ is the list of consecutive indices $i$ for which $g_i(w_k) = 0$ (the $i$ satisfying the prior inequalities), then the multiplicities of $w_k$ as a zero are:\ for $i_1,i_m$ the multiplicity is one; for $i_2,i_{m-1}$ the multiplicity is two; and so on.\label{mult-pattern}
\end{enumerate}

In other words, the multiplicity of $w_k$ as a zero of $g_{i}$ for $i = 1,2,3,\dotsc$ has the form 
$$
[0,0,\dotsc,0,1,2,3,\dotsc,3,2,1,0,0,\dotsc],
$$
where the first string of zeros ends at index $d(k,\rhobar)$ and there are $(d^{\new}(k,\rhobar)-1)$-many non-zero numbers.

Having defined the polynomials $g_{i}(w)$, we set 
\begin{equation*}
G_{\bar\rho}(w,t) = 1 + \sum_{i\geq 1} g_{i}(w)t^i \in 1 + t\Zp[\![w,t]\!].
\end{equation*}
The ghost conjecture predicts that this series models the characteristic power series of the $U_p$-operator acting on overconvergent $p$-adic cuspforms whose eigensystem modulo $p$ matches $\rhobar$.

More precisely, let $D \subseteq G_{\Q}$ be a decomposition group at $p$ and $I \subseteq D$ be the inertia subgroup. Write $\omega: G_{\Q} \rightarrow \Fp^\times$ for the  cyclotomic character modulo $p$. Then, denote by $b(\rhobar)$ the unique integer $2 \leq b(\rhobar) \leq p$ such that $\det(\rhobar)|_I  = \omega^{b(\rhobar)-1}$.

We consider the rigid analytic space $\mathcal W_{b(\rhobar)}$ whose $\Qpbar$-points are the $p$-adic weights $\kappa: \Z_p^\times \rightarrow \Qpbar^\times$ whose action on the the torsion subgroup of $\Z_p^\times$ is raising to the $b(\rhobar)$-th power. Given $\kappa \in \mathcal W_{b(\rhobar)}$, write $w_\kappa = \kappa(1+2p) - 1$ and $S_{\kappa}^{\dagger}(\rhobar)$ for the $\rhobar$-isotypic component of the space of overconvergent $p$-adic cuspforms of level $\Gamma_1(N)$ and weight $\kappa$. The zeros $w_k$ defined in (1) occur only at integer weights $k \geq 2$ that lie in $\mathcal W_{b(\rhobar)}$.  (Recall, $\det(\rhobar_f)|_I = \omega^{k-1}$, so if $d(k,\rhobar) \neq 0$ then $b(\rhobar) \equiv k \bmod p-1$.)

The ghost conjecture for $\rhobar$ stated in \cite{BergdallPollack-GhostPaper-II} is:

\begin{conj}[False]\label{conj:old-ghost} Assume $p \geq 5$ and $\rhobar \not\cong \omega^j \oplus \omega^{j+1}$ for any $j$.  If $\rhobar|_D$ is reducible, then, for each $\kappa \in \mathcal W_{b(\rhobar)}$, the Newton polygon of $G_{\bar\rho}(w_\kappa,t)$ is the same as the Newton polygon of the characteristic series of the $U_p$-operator acting on $S_{\kappa}^{\dagger}(\bar\rho)$.
\end{conj}

\subsection{Counter-examples}\label{subsec:counter-example}
We give two counter-examples here.

\subsubsection{Detailed example}\label{subsubsec:detailed}
Let $p = 5$ and $N = 3$. We choose $\bar\rho = \omega \oplus \omega\chi$ where $\chi$ is the quadratic Dirichlet character of conductor 3. The dimension functions $d$, $d_p$, and $d_p^{\new}$ can be determined using a computer algebra system with built-in libraries for modular forms, such as Magma \cite{MagmaCite} or Sage \cite{sagemath}. The result is compiled in Table \ref{table:isotypic-components}. From the dimensions, we see that 
\begin{align*}
g_{1}(w) &= w-w_3;\\
g_{2}(w) &= w-w_7;\\
g_{3}(w) &= (w-w_{11})(w-w_{15})\dotsb.
\end{align*}
Since $g_{1}(w_7)$ has $5$-adic valuation 1, the least slope on the Newton polygon of $G_{\rhobar}(w_7,t)$ is 1 or less. On the other hand, any weight 7 overconvergent $p$-adic eigenform with slope at most 1 is classical by Coleman's classicality theorem \cite[Theorem 6.1]{Coleman-ClassicalandOverconvergent}. Yet, the slopes of $U_5$ acting on $S_7^0(\rhobar) \subseteq S_7^{\dagger}$ are $[\frac{5}{2},\frac{5}{2},3,3]$ (calculated using Magma). This is the contradiction.

\begin{table}[htp]
\renewcommand{\arraystretch}{1.1}
\caption{Dimensions of $\rhobar$-isotypic components in $S_k(\Gamma_1(3))$ and $S_k(\Gamma_1(3)\cap \Gamma_0(5))$, for $\rhobar = \omega \oplus \omega\chi$.}
\begin{center}
\begin{tabular}{|c|c|c|c|}
\hline
$k$ & $d(k,\bar\rho)$ & $d_p(k,\bar\rho)$ & $d_p^{\new}(k,\bar\rho)$\\
\hline
3 & 0 & 2 & 2\\
7 & 1 & 4 & 2\\ 
11 & 2 & 6 & 2\\
15 & 2 & 8 & 4\\
$\vdots$ & $\vdots$ & $\vdots$ & $\vdots$\\
\hline
\end{tabular}
\end{center}
\label{table:isotypic-components}
\end{table}%

\begin{remark}\label{rmk:code-remark}
Strictly speaking, neither Magma (nor Sage) currently has intrinsic commands to determine the slopes in $S_k(\rhobar)$. The authors provide code in the github repository \cite{BergdallPollack-RhobarCode}.
\end{remark}

\subsubsection{Another example}
The reader might wonder if the fundamental issue in the prior example is that either $\rhobar$ is Eisenstein or $\rhobar \cong \rhobar \otimes \chi$, both somehow global phenomena. This is not the case. 

Explicitly, continue to let $p = 5$ and also let $\psi$ be the quadratic character modulo $43$. Then, in $S_7(\Gamma_1(43),\psi)$ there is a 20-dimensional Galois orbit of newforms in which the $5$-adic slope $2$ occurs twice, once for an absolutely irreducible $\rhobar$ and once more for $\rhobar \otimes \psi$ (which is a new representation). Further, $S_7(\rhobar)$ is 1-dimensional and thus 2 is the lowest slope occurring for $\rhobar$. Yet, as in Section \ref{subsubsec:detailed}, one can check that the ghost series formalism would predict the lowest slope in $S_7(\rhobar)$ is 1.

In fact, $\rhobar \otimes \omega^{-1}$ lifts to an ordinary eigenform of weight $5$, and thus $\rhobar$ is locally reducible at $p = 5$.  By Proposition \ref{prop:equiv-regular} and point (ii) in the proof of Theorem \ref{thm:equivalence} below, we even have
\begin{equation}\label{eqn:extension}
\rhobar|_D \cong \begin{pmatrix} \nr(\alpha)\omega & \ast \\ 0 & \nr(-\alpha)\omega \end{pmatrix}
\end{equation}
where $\alpha \in \overline{\F}_5^\times$. (For unfamiliar notations or definitions, see Section \ref{sec:regularity-define}.) The extension \eqref{eqn:extension} is even non-split! If it were split, then by \cite[Theorem 4.5]{Edixhoven-Weights} the representation $\rhobar \otimes \omega^{-1}$ would arise from a weight one form over $\overline{\F}_5$ of level 43 and character $\psi$. We verified such a form does not exist using a computer package called {\tt Weight1}, which calculates weight one eigenforms modulo $p$, written by Wiese. (The code, which requires Magma, is currently available on Wiese's website \cite{Wiese-WeightOnePackage}.)

\subsubsection{The general phenomenon}
The common link between the examples is  indicated by the equation \eqref{eqn:extension}. In both cases, $\rhobar \otimes \omega^{-1}$ is reducible at $p=5$ and its semi-simplification is an unramified representation for which a Frobenius element acts with trace zero. Generalizing these examples, the fundamental gap we found in the ghost conjecture occurs for $\rhobar$ that are twists of representations which locally at $p$ are reducible, with semi-simplification that is unramified with Frobenius trace zero. Indeed, in Proposition \ref{prop:theory-evidence} we will show that the ghost conjecture will always fail in such examples. However, as long as we exclude such representations (which we do in Section \ref{sec:regularity-define}), we believe that the ghost conjecture will hold for remaining $\rhobar$'s.

\subsection{A brief history}
We first encountered the example in Section \ref{subsubsec:detailed} in 2015-16, when we learned by explicit computation that the ghost series formalism failed to predict the 5-adic slopes of eigenforms of level $\Gamma_1(3)$ (with no restriction on the Galois representations modulo 5). We apparently overlooked that computation when formulating the $\rhobar$-version of our conjecture. Later, in 2018, James Newton asked us about the $\rhobar$-conjecture for examples where $\rhobar$ is unramified at $p$ (up to a twist). We mistakenly missed the counter-examples once more. We realized our mistake only in early 2021 after our attention was diverted toward questions on reductions of crystalline Galois representations.

\subsection{Plan for the remainder of the article}

In Section 2 we revisit the reducible versus irreducible dichotomy and replace it with a new one, which we name regular versus irregular. We study this new dichotomy in both local and global terms, relating it to crystalline lifts on the one hand and slopes of modular forms in low weight on the other. The salvaged conjecture and evidence is given Section 3. Concluding, and open-ended, remarks occupy Section 4.

\subsection{Acknowledgements}

We thank Liang Xiao for keeping us updated on the progress of his collaboration with Liu, Truong, and Zhao, which aims to prove the ghost conjecture. Both authors also thank the Max Planck Institute for Mathematics in Bonn, Germany, for its hospitality in July 2021, when this paper was largely written. The research was also supported by a Simons Collaboration Grant (PI:\ J.B.\ Award \#713782) and an NSF grant (PI:\ R.P.\ DMS-1702178).

\section{Salvaging the ghost conjecture:\ redefining regularity}\label{sec:regularity-define}

To salvage Conjecture \ref{conj:old-ghost}, we redefine regularity, replacing the reducible vs.\ irreducible dichotomy with a slightly different one. In this section we let $p$ be any prime, unless noted. 

We first fix notations for local Galois representations modulo $p$. Let $\Q_{p^2}$ be the unramified quadratic extension of $\Q_p$ and write $\omega_2 : G_{\Q_{p^2}} \rightarrow \overline\F_{p}^\times$ for a niveau 2 fundamental character. Let $G_{\Qp} = \Gal(\Qpbar/\Qp)$ and $G_{\Q_{p^2}}  = \Gal(\Qpbar/\Q_{p^2}) \subseteq G_{\Qp}$. We then write $\ind(\omega_2^s) = \Ind_{G_{\Q_p^2}}^{G_{\Q_p}}(\omega_2^s\eta)$ where $\eta$ is the quadratic character of $\Q_{p^2}$. Thus $\ind(\omega_2^s)$ has determinant $\omega^s$ and $\ind(\omega_2^s)|_I = \omega_2^{s} \oplus \omega_2^{sp}$. Given $\alpha \in \Fpbar^\times$ we also write $\nr(\alpha)$ for the character on $G_{\Q_p}$ that is unramified and whose value on a Frobenius element is $\alpha$. In this notation, $\ind(1) = \nr(\sqrt{-1}) \oplus \nr(-\sqrt{-1})$.

The continuous, semi-simple, representations $\overline r: G_{\Q_p} \rightarrow \GL_2(\overline \F_p)$ are completely classified. Up to a twist by a character, any $\overline r$ is isomorphic to either $\ind(\omega_2^s)$ for some $0 \leq s \leq p-1$ or $\nr(\alpha) \oplus \nr(\beta)\omega^t$ with $\alpha,\beta \in \Fpbar^\times$ and $0 \leq t \leq p-2$. The irreducible $\overline r$ occur for twists of $\ind(\omega_2^s)$ with $1 \leq s \leq p-1$. The representation $\ind(\omega_2^s)$ is reducible when $s \equiv 0 \bmod p+1$.

\begin{defn}\label{defn:irregular-local}
A continuous, semi-simple, representation $\overline r: G_{\Q_p} \rightarrow \GL_2(\overline \F_p)$ is called {\em irregular} if $\overline r$ is a twist of $\ind(\omega_2^s)$ for some $s \in \Z$.
\end{defn}

Naturally, we say $\overline r$ is regular when it is not irregular. Note that regular representations are always reducible, but the regularity of $\overline r$ depends on the entire action of $D$, whereas the reducibility depends only on $\overline r|_I$. 

The next proposition illustrates the regular versus irregular dichotomy in terms of crystalline lifts with small Hodge--Tate weights. The reader who is unfamiliar with $p$-adic Hodge theory may prefer to skip to Theorem \ref{thm:equivalence} or reference the well-written introduction to \cite{BuzzardGee-SmallSlope}. The notation is as follows. For $a_p \in \Qpbar$ such that $v_p(a_p) >0$ and $k\geq 2$, $V_{k,a_p}$ denotes the unique two-dimensional, irreducible, crystalline representation of $G_{\Qp}$ whose Hodge--Tate weights are $\{0,k-1\}$ and such that the characteristic polynomial of the crystalline Frobenius acting on $D_{\crys}(V_{k,a_p}^{\ast})$ is $t^2 - a_p t + p^{k-1}$. Write $\overline V_{k,a_p}$ for the semi-simplification of the reduction of $V_{k,a_p}$ modulo $p$.

\begin{prop}\label{prop:equiv-regular}
Let $\overline r: G_{\Qp} \rightarrow \GL_2(\Fpbar)$ be a continuous, semi-simple, representation. The following conditions are equivalent (the (a)'s and (b)'s are individually equivalent):
\begin{enumerate}
\item The representation $\overline r$ is irregular:
\begin{enumerate}
\item $\overline r$ is a twist of $\ind(\omega_2^s)$ with $s \not\equiv 0 \bmod p+1$, or
\item $\overline r$ is a twist of $\ind(\omega_2^s)$ with $s \equiv 0 \bmod p+1$.
\end{enumerate}

\item Either:
\begin{enumerate}
\item $\overline r$ irreducible, or
\item $\overline r$ a twist of an unramified representation whose Frobenius trace is zero.
\end{enumerate}
\item Either:
\begin{enumerate}
\item $\overline r$ is a twist of $\overline V_{k,a_p}$ for some $2 \leq k \leq p+1$ and $v_p(a_p) > 0$, or
\item $\overline r$ is a twist of $\overline V_{p+2,a_p}$ for any $v_p(a_p) > 1$.
\end{enumerate}
\end{enumerate}
\end{prop}
\begin{proof}
The equivalence of (1) and (2) follows directly from the discussion prior to Definition \ref{defn:irregular-local}. The equivalence of (1) and (3) follows from calculations of $\overline V_{k,a_p}$ when $k$ is small. Namely, if $2 \leq k \leq p+1$ then $\overline V_{k,a_p} \cong\ind(\omega_2^{k-1})$ (which is, in particular, irreducible and thus irregular). When $k=p+2$, we have:
\begin{equation}\label{eqn:Vp+2}
\overline V_{p+2,a_p} = \begin{cases}
\ind(\omega_2^2) & \text{if $0 < v_p(a_p) < 1$;}\\
\nr(\alpha) \oplus \nr(\alpha^{-1}) & \text{if $1 = v_p(a_p)$, where $\alpha+\alpha^{-1} = \frac{a_p}{p} \in \Fpbar^\times$;}\\
\ind(1)\otimes \omega&\text{if $1 < v_p(a_p)$.}
\end{cases}
\end{equation}
Thus $\overline V_{p+2,a_p}$ is irregular when $v_p(a_p) \neq 1$. See \cite[Th\'eor\`eme 3.2.1]{Berger-Modp-compat} for these results. From them, we plainly see (3) implies (1) and (1) implies (3) since any irregular $\overline r$ is a twist of $\ind(\omega_2^s)$ for some $0 \leq s \leq p-1$.
\end{proof}

\begin{remark}
The trivial representation mod 2 is irregular by Proposition \ref{prop:equiv-regular}(2), since it is unramified with Frobenius trace $2 = 0$. See Section \ref{sec:small-primes}.
\end{remark}

We now turn toward global considerations. Given a representation $\rho$, write $\rho^{\operatorname{ss}}$ for the semi-simplification. (This depends on the group that $\rho$ is representing, but the context will be clear any time we use the notation.) Since the decomposition groups $D \subseteq G_{\Q}$ at $p$ are all conjugate and isomorphic to $G_{\Qp}$, we can speak unambiguously about a global representation $\overline \rho$ being regular on $D$ or not:

\begin{defn}\label{defn:regular-global}
We say $\rhobar$ is regular if $(\rhobar|_D)^{\operatorname{ss}}$ is regular.
\end{defn}

The action of $D$ is reducible for every regular $\rhobar$, so our new regular condition implies the condition in Conjecture \ref{conj:old-ghost}. The difference is that we have further excluded one particular twist-stable family of reducible representations. 

We describe now a mechanism  in terms of the arithmetic of modular forms to determine whether or not a given $\rhobar$ is regular.  Write
$$
S_k(\!(\overline \rho)\!) = \sum_{j=0}^{p-2} S_{k}(\overline \rho \otimes \omega^j)  \subseteq S_k(\Gamma_1(N),\Qpbar).
$$
The space $S_k(\!(\overline \rho)\!)$ is stable under the $T_p$-operator. Recall, if $f$ is an eigenform then $v_p(a_p(f)) \geq 0$. We say $f$ is ordinary if $v_p(a_p(f)) = 0$.

\begin{thm}\label{thm:equivalence}
Consider the following two conditions:
\begin{enumerate}
\item The representation $\overline \rho$ is regular.
\item For all $2 \leq k \leq p+1$, the eigenforms in $S_{k}(\!(\overline \rho)\!)$ are all ordinary, and the eigenforms in $S_{p+2}(\!(\rhobar)\!)$ are either ordinary or satisfy $v_p(a_p(f)) = 1$.
\end{enumerate}
Then, (1) always implies (2). If $p \geq 3$, then (2) implies (1).

More precisely:
\begin{enumerate}[label=(\alph*)]
\item If $\rhobar|_D$ is irreducible, then there exists a non-ordinary eigenfrom $g$ of weight between $2$ and $\frac{p+3}{2}$ such that $\rhobar$ is a cyclotomic twist of $\rhobar_g$.
\item If $p \geq 3$ and $\rhobar$ is irregular but $\rhobar|_D$ is reducible, then there exists a $j$ such that $S_{p+2}(\rhobar\otimes \omega^j) \neq \{0\}$ and $v_p(a_p(g)) > 1$ for all $g \in S_{p+2}(\rhobar \otimes \omega^j)$.
\label{part:thm:equivalence-red}
\end{enumerate}
\end{thm}

\begin{proof}
To start, we recall two fundamental facts. For any eigenform $f$, we let
$$
\rho_f : G_{\Q} \rightarrow \GL_2(\Qpbar)
$$
be the corresponding global $p$-adic Galois representation. Then:

\begin{enumerate}[label=(\roman*)]
\item If $f$ is ordinary, then $\rho_f|_D$ is reducible and $(\rhobar_f|_I)^{\operatorname{ss}} \cong 1 \oplus \omega^{k-1}$.
\item  If $f$ is non-ordinary, then $\rho_f|_D$ is isomorphic to a twist of $V_{k,a_p(f)}$.
\end{enumerate}
It seems the first result was first proven by Deligne in the 1970's but never published. A common reference is \cite[Theorem 2]{Wiles-OrdinaryModular}. For the second, see \cite[Th\'eor\`eme 6.5]{Breuil-SomeRepresentations2}.

Now we prove (1) implies (2). Suppose that $\rhobar$ is regular and $f$ is an eigenform of weight $2 \leq k \leq p+2$ such that $\rhobar_f = \rhobar \otimes \omega^j$. Then, either $f$ is ordinary or $\overline V_{k,a_p(f)}$ is regular by (ii). By Proposition \ref{prop:equiv-regular}, and the classification of $\overline{V}_{p+2,a_p}$ in equation \eqref{eqn:Vp+2}, the latter can only happen if $k = p+2$ and $v_p(a_p(f)) = 1$. This proves (1) implies (2).

Now we will prove (2) implies (1). By \cite[Theorem 3.4]{Edixhoven-Weights} there exists an eigenform $f$ of weight $1 \leq k \leq p+1$ where $\rhobar \cong \rhobar_f \otimes \omega^j$, with the caveat that $f$ may be an Eisenstein series or, in the case $k=1$, $f$ may be only a mod $p$ modular form.  Furthermore, the theory of $\theta$-cycles (see \cite[Proposition 3.3]{Edixhoven-Weights}) predicts the least positive weight in which $\rhobar_f \otimes \omega^j$ lifts to a modular eigenform.

Now consider (a), where $\rhobar|_D$ is irreducible. Then, the possibility of $k = 1$ cannot actually occur. Indeed, if $f$ has weight one then $\rhobar_f|_D$ is unramified by theorems of Gross, Coleman--Voloch, and Wiese. See \cite[Corollary 1.3]{Wiese-WeightOne}. The theory of $\theta$-cycles ($a_p = 0$ and $k \geq 2$ in \cite[Proposition 3.3]{Edixhoven-Weights} cited above) then implies there exists a non-ordinary eigenform $g$ of weight between $2$ and $\frac{p+3}{2}$ such that $\rhobar_g \cong \rhobar_f \otimes \omega^j$ for some $j$. This completes the proof of (a). 

 In case (b), we assume that $p \geq 3$ and after twisting that  $\rhobar = \rhobar_f$ with $\rhobar_f|_D$ reducible, yet irregular and $f$ has weight $1 \leq k \leq p+1$. We first claim that either:\ $k=p$ and $f$ is ordinary, or $k=1$.  Further, in both cases $(\rhobar_f|_I)^{\operatorname{ss}} \cong 1 \oplus 1$. 

To see this, note first that if $k=1$ then, as above, $\rhobar_f$ is unramifed at $p$ and thus $(\rhobar_f|_I)^{\operatorname{ss}} \cong 1 \oplus 1$.  If $k \geq 2$, we first check that $f$ is ordinary. Indeed, if $f$ is Eisenstein, then it is ordinary because $p \nmid N$. If $f$ were cuspidal non-ordinary, then (ii) and Proposition \ref{prop:equiv-regular} would imply $\rhobar_f|_D$ is irreducible, which it is not. So, $f$ is ordinary. Then, by (i) we see $(\rhobar_f|_I)^{\operatorname{ss}} \cong 1 \oplus \omega^{k-1}$. But $\rhobar$ is irregular, and so Proposition \ref{prop:equiv-regular} implies $\omega^{k-1} = 1$, which implies $k = p$ since $2\leq k \leq p+1$.

Finally, regardless of whether $k = p$ and $f$ is ordinary, or $k=1$, the theory of $\theta$-cycles guarantees there exists a modular eigenform $g$ of weight $p+2$ such that $\rhobar_g \cong \rhobar_f \otimes \omega$. For such $g$, we have $(\rhobar_g|_I)^{\operatorname{ss}} \cong \omega \oplus \omega$ and thus $g$ is non-ordinary by (i).  Here we use $p \neq 2$ to know $\omega \neq 1$. In particular we've shown that $S_{p+2}(\rhobar_f \otimes \omega) \neq \{0\}$. We also see from (ii) that $\overline V_{p+2,a_p(g)}$ is reducible and irregular for any $g \in S_{p+2}(\rhobar_f \otimes \omega)$. Thus, by Proposition \ref{prop:equiv-regular}, $v_p(a_p(g)) > 1$ for all $g \in S_{p+2}(\rhobar_f \otimes \omega)$. This completes the proof of (b).
\end{proof}

\begin{remark}
A key step in the proof of part \ref{part:thm:equivalence-red} of Theorem \ref{thm:equivalence} is that if $f$ is an ordinary eigenform of weight $p$, then $\rhobar_f \otimes \omega$ lifts to an eigenform of weight $p+2$. This statement also holds when $f$ is non-ordinary, as long as $p \geq 5$. For instance, one could directly argue using modular symbols modulo $p$ as in \cite[Theorem 3.4(a,b)]{AshStevens-Duke}. (That reference assumes $p \geq 5$ throughout.)

For $p = 3$, though, there exist weight 3 forms $f$ where $\rhobar_f \otimes \omega$ doesn't lift to weight 5. For instance, if $\chi$ is the quadratic character modulo 7, then $S_3(\Gamma_1(7),\chi)$ is spanned by a unique eigenform (\cite[Newform 7.3.b.a]{LMFDB}) whose $q$-expansion begins $q - 3q^2 + \dotsb$, whereas $S_5(\Gamma_1(7),\chi)$ is spanned by a unique eigenform (\cite[Newform 7.5.b.a]{LMFDB}) whose $q$-expansion is $q + q^2 + \dotsb$. These two eigenforms are not twists of each other modulo 3.
\end{remark}

\section{The salvaged conjecture and evidence}

We now state for the record our salvaged $\rhobar$-ghost conjecture. We place into the conjecture the hypotheses we need, as a reminder.

\begin{conj}[The salvaged $\rhobar$-ghost conjecture]\label{conj:ghost}
Let $p \geq 5$. 
If $\rhobar$ is regular in the sense of Definition \ref{defn:regular-global}, then, for each $\kappa \in \mathcal W_{b(\rhobar)}$, the Newton polygon of $G_{\bar\rho}(w_\kappa,t)$ is equal to the Newton polygon of the characteristic series of the $U_p$-operator acting on $S_{\kappa}^{\dagger}(\bar\rho)$.
\end{conj}
In the next subsections, we provide computational and theoretical evidence for our salvaged conjecture. 

\begin{remark}
Notice that Conjecture \ref{conj:ghost} includes $\rhobar \cong \omega^j \oplus \omega^{j+1}$, in contrast with Conjecture \ref{conj:old-ghost}. See Section \ref{subsec:Galois-mult} below.
\end{remark}

\begin{remark}
Condition (2) in Theorem \ref{thm:equivalence} appears in Buzzard's slope conjectures  when $p = 2$ (see \cite[Definition 1.3]{Buzzard-SlopeQuestions}). For odd primes, it was irrelevant in Buzzard's work, since restricting to level $\Gamma_0(N)$ removes concerns about odd weights $p$ and $p+2$.
\end{remark}

\subsection{Galois multiplicities}\label{subsec:Galois-mult}
In order to construct the series $G_{\rhobar}(w,t)$, and thus test Conjecture \ref{conj:ghost}, one must understand the functions $d(k,\rhobar)$ and $d_p(k,\rhobar)$. In fact, each of these functions are determined after a finite computation. Specifically, either can be determined from the values $d(k',\rhobar \otimes \omega^i)$ with $2 \leq k' \leq p+1$ and $0 \leq i \leq p-2$ along with multiplicities of Eisenstein series (if $\rhobar$ is Eisenstein). This is explained in \cite[Section 6]{BergdallPollack-GhostPaper-II}. The technique involves modular symbols modulo $p$, which is the source of the assumption $p \geq 5$. We do not currently know dimension formulas when $p=2,3$. For an alternative method via the trace formula, which is adaptable to $p=2,3$, see forthcoming work of Anni, Ghitza, and Medvedovsky \cite{AnniGhitzaMedvedovsky-Dimensions}.

There is one caveat:\ the exposition in \cite{BergdallPollack-GhostPaper-II} ignores the case $\rhobar \cong \omega^j \oplus \omega^{j+1}$ (``twists of $E_2$'') for convenience. These are precisely the Galois representations modulo $p$ that contribute to torsion in $H^2_c$'s, which creates complications. However, the complications are luckily limited to exposition. The actual result, that low weight calculations for $\rhobar$ and all its twists will determine $\rhobar$-multiplicities in all weights, are still valid for twists of $E_2$. 

So, while omitting more specific details, we used this principle in order to carry out the numerical testing explained in the next section, regardless of whether $\rhobar$ is a twist of $E_2$ or not. (But still, $p \geq 5$.) For the interested reader, we also created computer code that implements the natural process to calculate $d(k,\rhobar)$  from low weight values.  See \cite{BergdallPollack-RhobarCode}.

\subsection{Numerical evidence}\label{subsec:numerical}
In \cite[Section 7.2]{BergdallPollack-GhostPaper-II}, we gave significant numerical evidence for Conjecture \ref{conj:ghost}. Here we report more numerical evidence, especially highlighting the following two cases:
\begin{enumerate}
\item Representations $\rhobar: G_{\Q} \rightarrow \GL_2(\Fpbar)$ with $(\rhobar|_D)^{\operatorname{ss}}$ a twist of unramified with {\em non-zero} Frobenius trace (Examples \ref{exam:one_plus_chi}-\ref{exam:p=5-N=23}).
\item Twists of $1 \oplus \omega$ (Example \ref{exam:E2}).
\end{enumerate}
The latter was a case not numerically tested in \cite{BergdallPollack-GhostPaper-II}, and the former is meant to ensure the issue with the counter-examples in Section \ref{subsec:counter-example} has been accurately diagnosed.

In each example, we describe a single $\rhobar$ and explain why it is regular. Then, when we write that we ``verified Conjecture \ref{conj:ghost} in this case for weights $k$...'' we mean:\ for each $\rhobar' = \rhobar\otimes \omega^i$ we calculated the $T_p$-slopes on $S_k(\rhobar')$ and checked that they matched the first $d(k,\rhobar')$-many slopes on $G_{\rhobar'}(w_k,t)$. That is, we checked (using the code referenced in Remark \ref{rmk:code-remark}) the ghost series accurately predicts the classical slopes for $\rhobar$ and each of its twists (for a range of weights).

\begin{ex}\label{exam:one_plus_chi}
Let $\chi$ denote the quadratic character of conductor $N$ thought of as taking values in $\F_p^\times$ and let $\rhobar = 1 \oplus \chi$.  If $p \nmid N$ then $\rhobar$ is unramified at $p$ and it is regular if and only if $\chi(p) = 1$.  For each of $N=3,7,11$ and for all primes $5 \leq p \leq 97$ for which $\rhobar$ is regular, we verified Conjecture \ref{conj:ghost} in this case for weights $k \leq 385, 173, 191$ respectively.  
\end{ex}

\begin{ex}\label{exam:p=5-N=23}
Let $p = 5$, $N = 23$, and $\chi$ be the quadratic character modulo 23. Then, $S_5(\Gamma_1(23),\chi)$ supports a unique $\rhobar$ whose Frobenius trace at $\ell = 2$ is equal to 2. Specifically, if $f$ is a member of the newform orbit \cite[Newform 23.5.b.b]{LMFDB} and $\mathfrak p$ is the unique prime above 5 with inertia degree 2 for the number field generated by $f$'s Hecke eigenvalues, then $\rhobar$ is the Galois representation associated with $f$ modulo $\mathfrak p$. The eigenform $f$ is $\mathfrak p$-ordinary and $\rhobar \otimes \omega$ arises from a weight 7 form with slope 1. By Theorem \ref{thm:equivalence}, $\rhobar$ is regular. Unlike the previous example, this $\rhobar$ is globally irreducible.  We verified Conjecture \ref{conj:ghost} in this case for weights $k\leq 155$. 
\end{ex}

\begin{ex}\label{exam:E2}
Let $\rhobar = 1 \oplus \omega$ which is regular.  For both of $N=1$ and $11$ and for all primes $p$ with $5 \leq p \leq 97$, we verified Conjecture \ref{conj:ghost} in this case for weights $k \leq 750, 172$ respectively.  
\end{ex}

\begin{remark}
Gouv\^ea and Mazur once conjectured that if $h$ is a rational number and $k,k' \geq 2h+2$, then the multiplicity of the ($T_p$-)slope $h$ in weight $k$ will be the same as the multiplicity in weight $k'$ provided $k \equiv k' \bmod (p-1)p^{\lceil h\rceil }$ (see \cite[Conjecture 1]{GouveaMazur-FamiliesEigenforms}). Buzzard and Calegari found many counter-examples, all occuring in situations where $\rhobar$ is locally irreducible at $p$ (\cite{BuzzardCalegari-GouveaMazur}). Those $\rhobar$ are, in particular, irregular. In the final paragraphs of {\em loc. cit.}, however, it is mentioned that Clay noticed issues in a globally reducible, and in fact regular, situation.

Specifically, let $\rhobar = \omega + \omega\chi$ where $\chi$ is the quadratic character modulo 4, which is regular for $p = 5$ because $\chi(5) = \chi(1) = 1$ (compare with Example \ref{exam:one_plus_chi}). The violation of Gouv\^ea--Mazur, which Clay noticed, occurs because the $5$-adic $\rhobar$-slopes in weight $k=7$ at level $\Gamma_1(4)$ are $\{1,1\}$ (two times), whereas in weight $k=27$ they are $\{1,1,1,1\}$ (four times). We highlight this because this violation of the Gouv\^ea--Mazur conjecture is consistent with Conjecture \ref{conj:ghost}. That is, the ghost series predicts exactly the same slope pattern.

The data collected in Example \ref{exam:one_plus_chi} gives rather systematic counter-examples to the Gouv\^ea--Mazur conjecture  (on $\rhobar$-subspaces), all of which are consistent with Conjecture \ref{conj:ghost}. The issue is also a local one, as opposed to a global one, it seems. We found the same phenomenon in the setting of Example \ref{exam:p=5-N=23}, which is globally irreducible. Specifically, in that example, in weight $7$ the slope sequence is $\{1\}$, in weight 27 it is $\{1,1\}$. Conjecture \ref{conj:ghost} predicts these slopes.
\end{remark}

\subsection{Theoretical evidence}

For theoretical evidence, we show that $\rhobar$ must be regular if the ghost series correctly calculates slopes of modular forms.

\begin{prop}\label{prop:theory-evidence}
Let $p\geq 5$. If $\rhobar$ is irregular, then the ghost series of some twist of $\rhobar$ does not correctly compute slopes of modular forms.  That is, there is some $j$ and some weight $\kappa \in \mathcal W_{b(\rhobar\otimes \omega^j)}$ such the slopes of the Newton polygon of $G_{\rhobar \otimes \omega^j}(w_\kappa,t)$ do not match the slopes of $U_p$ acting on $S_{\kappa}^{\dagger}(\rhobar  \otimes \omega^j)$.
\end{prop}

\begin{proof}
First assume that $\rhobar|_D$ is irreducible. By Theorem \ref{thm:equivalence}, after replacing $\rhobar$ by a twist, we may assume that $\rhobar$ arises from a non-ordinary cuspform of weight $k$ between $2$ and $\frac{p+3}{2}$.   We will show that the ghost series for $\rhobar$ does not correctly compute the slopes of $U_p$ in weight $k=b(\rhobar)$.

To this end, write $G_{\rhobar}(w,t) = 1 +g_1(w) t + g_2(w) t^2 + \dotsb$.
The smallest $k$ such that $w_k$ is possibly a zero of any $g_i(w)$ is $k = b(\rhobar)$.  From the definition of the ghost series, if $w_{b(\rhobar)}$ is a zero is for $g_i(w)$ then 
$$
i >  d = d(b(\rhobar),\rhobar) = \dim S_{b(\rhobar)}(\rhobar).
$$
Since $p\geq 5$, we have that $n\mapsto d(b(\rhobar)+n(p-1),\rhobar)$ is increasing by \cite[Proposition 6.12]{BergdallPollack-GhostPaper-II}. So, $G_{\rhobar}(w,t) = 1 + \dotsb + t^d + \dotsb$ and thus the first $d$ slopes of the Newton polygon of this ghost series (at any weight!) are all 0.  In particular, if the slopes attached to this series equaled the true slopes of $U_p$, then the ($T_p$-)slopes in $S_{b(\rhobar)}(\rhobar)$ would be all be 0.  But this contradicts the fact that $\rhobar$ arises from a {\it non-ordinary} form of weight $b(\rhobar)$.

Now we consider the case where $\rhobar|_D$ is reducible (but still irregular).  By part \ref{part:thm:equivalence-red} of Theorem \ref{thm:equivalence}, after replacing $\rhobar$ by a twist, we may assume that $S_{p+2}(\rhobar) \neq \{0\}$ and $v_p(a_p(g)) > 1$ for all $g \in S_{p+2}(\rhobar)$.  We will show that the ghost series for $\rhobar$ does not correctly compute slopes in weight $p+2$.

To this end, again write $G_{\rhobar}(w,t) = 1 +g_1(w) t + g_2(w) t^2 + \dotsb$.  
Since $S_{p+2}(\rhobar)$ is non-zero, we have that $g_1(w_{p+2}) \neq 0$.  Thus, as argued above, $g_1(w) = (w-w_3)^{m}$, and $m \leq 1$ because the multiplicity pattern described on pg.\ \pageref{mult-pattern} begins $1,2,\dotsc$.  Thus, the lowest slope of the Newton polygon of $G_{\rhobar}(w_{p+2},t)$ is at most 1. But we've also assumed every slope in $S_{p+2}(\rhobar)$ is larger than one. This is a contradiction.
\end{proof}

\section{Conclusion}\label{sec:conclusion}

We have presented a new, salvaged, ghost conjecture for Galois representations modulo $p$, along with theoretical and computational evidence. We end by briefly describing two questions raised by our work (separate from whether or not the conjecture is true).

\subsection{Small primes}
\label{sec:small-primes}
Let $p = 2$. According to the definitions given here, the trivial representation $\overline \rho = 1 \oplus 1$ is irregular. This is rather confouding. The $2$-adic slopes at level $\SL_2(\Z)$, where the only Galois representation modulo 2 is the trivial one, are the most carefully studied of all examples. It is a major motivating example for both Buzzard's work and the work of the authors. The ghost formalism is believed to predict the $2$-adic slopes at level $\SL_2(\Z)$, yet it does not (nor does Buzzard's work) correctly predict the $2$-adic $\rhobar$-slopes at level $\Gamma_0(5)$ where $\rhobar$ is still the trivial representation.

Of course, the second condition of Theorem \ref{thm:equivalence} holds for the trivial representation when $p = 2$ and $N=1$, so perhaps that condition is the correct one for which the ghost series will predict slopes (regardless of $p$ being small). This is in-line with Buzzard's {\em ad hoc} definition of $\Gamma_0(N)$-regular when $p = 2$ in \cite{Buzzard-SlopeQuestions}. Note however that even in the $\Gamma_0(N)$-regular situation, the $2$-adic ghost conjecture requires a modification as in \cite[Section 5]{BergdallPollack-GhostPaperShort}.

The authors have not seen a reason to exclude $p=3$ from Conjecture \ref{conj:ghost}. Yet, we want to be careful. We have not done extensive testing for $p = 3$ because the derivations of Galois multiplicities in Section \ref{subsec:Galois-mult} relies on \cite{AshStevens-Duke}, which assumes that $p \geq 5$. The same methods work immediately when $p=2,3$ as long as $\Gamma_1(N)$ is a group without $p$-torsion, but stating a conjecture for $p = 3$ under such a restriction would be rather prosaic.

\subsection{Integral slopes and local Galois representations}
It is conjectured that if $p$ is odd, $k$ is even, and $v_p(a_p) \not\in \mathbb N$ then $\overline V_{k,a_p}$ is irreducible. See \cite[Conjecture 4.1.1]{BuzzardGee-Slopes}. The possibility of such a result was discussed ``in emails between Breuil, Buzzard, and Emerton'' in the early 2000's according to {\em loc.\ cit.} To give it a neutral name we'll call it the {\em integral slope conjecture}. We've been told by Breuil and Buzzard that the inspiration for discussing the integral slope conjecture was numerical calculations of global slopes and explicit local calculations in small weights. In light of this article's emphasis on reducibility being a faulty lens through which to predict global slopes, we offer the following comments. 

When $p$ is odd and $k$ is even, the irregular $\overline V_{k,a_p}$ in the sense of Definition \ref{defn:irregular-local} are all irreducible. So, the integral slope conjecture might as well say if $k$ is even and $v_p(a_p) \not\in \mathbb N$ then $\overline V_{k,a_p}$ is irregular.  This re-phrasing opens up the possibility of stating a more full conjecture. For instance, $p = 2$ is not included in the original integral slope conjecture because when $p = 2$ and $k \equiv 4 \bmod 6$, then $\overline V_{k,0}$ is reducible and hence so is $\overline V_{k,a_p}$ for any $v_p(a_p) \gg 0$ (see \cite[Remark 4.1.6]{BuzzardGee-Slopes}). Yet, these examples are in fact all irregular. Separate objections are raised when $k$ is odd (see \cite[Remark 4.1.5]{BuzzardGee-Slopes}), each of which is rendered moot if instead the conjecture begins with the assumption that $v_p(a_p) \not\in \frac{1}{2}\mathbb N$ and ends with the conclusion that $\overline V_{k,a_p}$ is irregular. (One might also note that in Proposition \ref{prop:equiv-regular} we could additionally add the condition (d):\ $\overline r$ is a twist of $\overline{V}_{k,0}$ for some $k$.) 

So, we know no counter-examples or counter-arguments for the statement:
\begin{equation}
\tag{$\star$}\label{eqn:irregular-conj}
\parbox{\dimexpr\linewidth-10em}{
\em ``For any prime $p$, if $k$ is even and $v_p(a_p) \not\in \N$ or if $k$ is odd and $v_p(a_p) \not\in \frac{1}{2}\N$, then $\overline V_{k,a_p}$ is irregular.''
}
\end{equation}
In fact, \eqref{eqn:irregular-conj} is supported by theorems of Buzzard and Gee \cite[Theorem 1.6]{BuzzardGee-SmallSlope} and Bhattacharya and Ghate \cite[Theorem 1.1]{BhattacharyaGhate-Slope12}. An example theorem derived from the latter work is that $\overline V_{k,a_p} \otimes \omega^{-1} \cong \ind(1)$, as long as $p \geq 5$ and $1 < v_p(a_p) < 2$ and $k \equiv p+2 \bmod p^2(p-1)$.  This nicely complements the discussion in the previous paragraph, further illustrating why ``irregular'' replacing ``irreducible'' is necessary for any generalization of the integral slope conjecture. We note, finally, that \eqref{eqn:irregular-conj} is also consistent with a theorem of Nagel and Pande \cite[Theorem 0.1]{NagelPande-Slopes23} focused on $2 < v_p(a_p) < 3$, and work of Arsovski \cite[Theorem 1]{Arsovski-Reductions1}, which covers a more sporadic range of slopes.

This article began by correcting a false conjecture. So, here, it is best to wait for a more conceptual explanation before declaring \eqref{eqn:irregular-conj} as an outright conjecture. Even the integral slope conjecture, limited to even weights and odd $p$, is based on a combination of computational evidence and theoretical evidence restricted to low weights and slopes. It has not, for instance, been contextualized within a wider context of all local Galois representations modulo $p$ nor, except in the above paragraphs, have issues with odd $k$ or $p = 2$ been really confronted. Does replacing reducibility by regularity offer any conceptual clarity? There's certainly a fascinating opportunity here for  further explanation.

\bibliography{ghost_oversight_bib}
\bibliographystyle{abbrv}

\end{document}